\documentclass[11pt,twoside]{amsart}

\usepackage{setspace, a4wide}
\usepackage[T1]{fontenc}
\usepackage[english]{babel}
\usepackage{amssymb,amsfonts,amsthm,amsmath}
\usepackage{layout,version,enumerate}
\usepackage[all,matrix]{xy}
\usepackage[varg]{pxfonts}
\usepackage{indentfirst}
\usepackage{url}
\usepackage[colorlinks=true,linkcolor=blue,citecolor=magenta, pagebackref]{hyperref}

\usepackage[numbers]{natbib}

\theoremstyle{plain}
\newtheorem{thm}{Theorem}[section]
\newtheorem{prop}[thm]{Proposition}

\newtheorem{lem}[thm]{Lemma}

\newtheorem*{thms}{Theorem}

\theoremstyle{definition}
\newtheorem{dfn}[thm]{Definition}

\theoremstyle{remark}

\newtheorem{rem}[thm]{Remark}

\newcommand{\bn}{\mathbb{N}}

\newcommand{\br}{\mathbb{R}}
\newcommand{\bc}{\mathbb{C}}
\renewcommand{\bf}{\mathbb{F}}

\newcommand{\ph}{\varphi}
\newcommand{\eps}{\varepsilon}
\newcommand{\id}{\mathfrak{id}}
\newcommand{\dist}{\mathrm{d}}

\newcommand{\rank}{\mathfrak{r}}
\renewcommand{\sl}{\mathrm{SL}}
\renewcommand{\id}{\mathrm{1}}

\usepackage{tikz}
\usetikzlibrary{matrix,arrows}

\begin{document}

\onehalfspace

\thispagestyle{empty}
\title{An exotic group as limit of finite special linear groups}

\author{Alessandro Carderi}
\address{A.C., Institut f\"ur Geometrie, TU Dresden, 01062 Dresden, Germany}
\email{alessandro.carderi@tu-dresden.de}

\author{Andreas Thom}
\address{A.T., Institut f\"ur Geometrie, TU Dresden, 01062 Dresden, Germany}
\email{andreas.thom@tu-dresden.de}

\date{\today}

\begin{abstract}
{ We consider the Polish group obtained as the rank-completion of an inductive limit of finite special linear groups. This Polish group is topologically simple modulo its center, it is extremely amenable and has no non-trivial strongly continuous unitary representation on a Hilbert space.}
\end{abstract}

\maketitle

\tableofcontents

\section*{Introduction}

Let $\bf_q$ be a finite field with $q=p^h$ elements and let $\sl_n(q)$ be the special linear group over $\bf_q$. We denote by $\rank(k)$ the \textit{rank} of a matrix $k\in M_n(\bf_q)$. We equip the groups $\sl_n(q)$ with the (normalized) \textit{rank-distance}, $\dist(g,h):=\frac{1}{n}\rank(g-h) \in [0,1]$. Note that $\dist$ is a bi-invariant metric, which means that it is a metric such that $\dist(gh,gk)=\dist(h,k)=\dist(hg,kg)$ for every $g,h,k \in \sl_n(q)$. For every $n\in\bn$, we consider the diagonal embedding \[\ph_n:\sl_{2^{n}}(q)\rightarrow \sl_{2^{n+1}}(q),\quad \text{ defined by } \quad \ph_n(g):=\begin{pmatrix} g&0\\0&g\end{pmatrix}. \]

Observe that for every $n$, $\ph_n$ is an isometric homomorphism. We denote by $A_0(q)$ the countable group arising as the inductive limit of the family $\{(\sl_{2^n}(q),\ph_n)\}_n$ and observe that we can extend the rank-metric $\dist$ canonically to $A_0(q)$. Let $A(q)$ be the metric-completion of $A_0(q)$ with respect to $\dist$, i.e., $A(q)$ is a Polish group and the natural extension of the rank-metric is complete and bi-invariant. The purpose of this note is to study the topological group $A(q)$. 

In many ways one can think of $A(q)$ as a finite characteristic analogue of the unitary group of the hyperfinite II$_1$-factor, which arises as a certain metric inductive limit of the sequence of unitary groups $U(2) \subset U(4) \subset U(8) \subset \cdots$, analogous to the construction above. However, the group $A(q)$ reflects at the same time in an intrinsic way asymptotic properties of the sequence of finite groups $\sl_n(q)$ -- in particular of the interplay of group structure, normalized counting measure, and normalized rank-metric. In this note we want to develop this analogy and prove various results that show similarities but also differences to the II$_1$-factor case. Some of the techniques that we are using are inspired by the theory of von Neumann algebras whereas others are completely independent.

\vspace{0.2cm}

Our main result is the following theorem.

\begin{thms}
  The Polish group $A(q)$ has the following properties:
  \begin{itemize}
  \item every strongly continuous unitary representation of $A(q)$ on a Hilbert space is trivial,
  \item the group $A(q)$ is extremely amenable,
  \item the center of $A(q)$ is isomorphic to $\bf_q^{\times}$ and the quotient by its center is topologically simple,
  \item it contains every countable amenable group and, in case $q$ is odd, the free group on two generators as discrete subgroups. 
  \end{itemize}
\end{thms}

Before we start recalling some of the concepts that are used in the statement of the theorem, let us make some further remarks. In a similar way, one can form an inductive limit of unital rings
$$\bf_q \subset M_2(\bf_q) \subset M_4(\bf_q) \subset \cdots \subset M_0(\bf_q),$$
where again, the normalized rank-metric allows to complete $M_0(\bf_q)$ to a complete von Neumann regular ring $M(\bf_q)$. Following von Neumann, we can view $M(\bf_q)$ as the coordinatization of a continuous geometry, see \cite{vonNeumann1960} for more details. 
The unit group of this ring is naturally isomorphic to $A(q)$. Indeed, any invertible element in $M(\bf_q)$ must be the limit of invertible elements and a rank-one perturbation takes care of the determinant. Note that the algebra $M(\bf_q)$ does not depend on the special choice of inclusions. Indeed, Halperin showed in \cite{Halperin} that (in complete analogy to the II$_1$-factor situation) any choice whatsoever yields the same algebra. We thank G\'abor Elek for pointing out Halperin's work.

\subsection*{Unitary representability}

We recall that a group is \textit{unitarily representable} if it embeds in the unitary group of a Hilbert space as a topological group and it is \textit{exotic} if every continuous unitary representation on a Hilbert space is trivial, see \cite{Banaszczyk1991}. While all second countable locally compact groups are unitarily representable -- via the regular representation -- there are several examples of Polish groups that are not unitarily representable. For example, the Banach space $\ell^p$ is unitarily representable if and only if $1\leq p\leq 2$, \cite{Megrelishvili2000}, see also \cite{Galindo2009, Gao2005}. The first example of an exotic group was found by Herer and Christensen in \cite{Herer1975} and a surprising result of Megrelishvili, \cite{Megrelishvili2001}, states that the group of orientation preserving homeomorphisms of the interval has no non-trivial unitary representation (not even a representation on a reflexive Banach space). However, most of the known examples of exotic groups are either abelian or do not have a compatible bi-invariant metric.

\subsection*{Amenability and extreme amenability}

A topological group is said to be \textit{amenable} if there exists an invariant mean on the commutative $C^*$-algebra of left-uniformly continuous complex-valued functions on the group. It is a standard fact that any topological group that admits a dense locally finite subgroup is amenable -- in particular, the group $A(q)$ is amenable for any prime-power $q$. See Appendix G of \cite{Bekka2008} for more details.

A topological group is said to be \textit{extremely amenable} if each continuous action of the group on a compact topological space admits a fixed point. See \cite{Pestov2006} for more details. Let us remark that every amenable exotic group is extremely amenable. Indeed, by amenability, any action on a compact space preserves a probability measure; since the Koopman representation of the action is by hypothesis trivial, the measure has to be a Dirac measure. The first example of an extremely amenable group was found by Herer and Christensen and they proved extreme amenability exactly by showing that the group is amenable and exotic. Nowadays, several examples of extremely amenable groups are known, for example the unitary group of a separable Hilbert space \cite{Gromov1993}, the automorphism group of a standard probability space \cite{Giordano2002}, and the full group of a hyperfinite equivalence relation \cite{Giordano2007}. However as for the previous property, most of the known examples do not have a compatible bi-invariant metric. 

\subsection*{Generalizations and Open problems}

The proof of the theorem is robust and can be generalized to other (inductive) sequences of non-abelian finite (quasi-)simple groups of increasing rank. 
Let us finish the introduction by listing a number of open problems that we find interesting and challenging.

\begin{itemize}
\item Is $A(q)$ contractible? 
\item Is $A(q)/\bf_q^{\times}$ simple?
\item Does it have a unique Polish group topology?
\item More generally, is every homomorphism from $A(q)$ to another Polish group automatically continuous?
\item Does it have (isometric) representations on a reflexive Banach space?
\item Does $A(q)/\bf_q^{\times}$ have the bounded normal generation property, see \cite{Dowerk2015b}?
\end{itemize}

The first question has a positive answer in the case of the hyperfinite II$_1$-factors by work of Popa-Takesaki \cite{Popa2014}, but we were unable to generalize the methods to our setting. A more detailed study of the more basic properties of the algebras $M(\bf_q)$ and their unit groups $A(q)$ will be subject of another study.

\section{No non-trivial unitary representations}

In this first section, we will prove that $A(q)$ has no non-trivial continuous representation on a Hilbert space, that is we will show that $A(q)$ is exotic.

\begin{dfn}
A complex continuous function $\psi$ on a Polish group $G$ is \textbf{positive definite} if $\psi(\id_G)=1$ and for every $a_1,\ldots,a_n\in\bc$ and all $g_1,\ldots,g_n\in G$, we have \[\sum_{i,j=1}^n a_i\bar{a}_j \psi(g_j^{-1}g_i)\geq 0.\]

A positive definite function $\chi$ is a \textbf{character} if it is conjugation invariant, that is for every $g,h\in G$, we have $\chi(hgh^{-1})=\chi(g)$. 
\end{dfn}

We will use some easy facts about positive definite functions which are covered in the Appendix C of \cite{Bekka2008}. For example the Cauchy-Schwarz inequality for positive functions implies that such functions have their maximal value at the identity so that any positive definite function is bounded in absolute value by $1$. We say that a positive definite function is \textit{trivial} if $\psi(g)=1$ for every $g\in G$. Every positive definite function $\psi$ also satisfies the following standard inequality, which can be found in \cite [Proposition \textbf{C.4.2}]{Bekka2008}: \begin{equation}\tag{$\star$} |\psi(g)-\psi(h)|^2\leq 2 (1-Re(\psi(g^{-1}h))).\label{eq:ine}\end{equation}

There is an important relation between positive definite functions and unitary representations: the GNS construction; any positive definite function gives rise to a unitary representation and the diagonal matrix coefficients of any unitary representation are positive definite functions. In particular, a group without any non-trivial continuous positive definite function has no non-trivial strongly continuous unitary representation on a Hilbert space, see Appendix C of \cite{Bekka2008} for more informations. 

The following useful proposition will be needed in the course of the proof.

\begin{prop}\label{prop:algclos} Let $n,m \in \mathbb N$.
The canonical inclusion $\sl_{2^n}(q)\hookrightarrow A_0(q)$ extends to an isometric homomorphism $\sl_{2^n}(q^{2^m})\hookrightarrow A_0(q)$.
\end{prop}
\begin{proof}
  For every $k$, the field $\bf_{q^{2^{k+1}}}$ is an algebraic extension of degree $2$ of $\bf_{q^{2^k}}$, hence there exists $\sigma_k\in \bf_{q^{2^{k+1}}}$ such that $\{1,\sigma_k\}$ is a $\bf_{q^{2^{k}}}$-base of $\bf_{q^{2^{k+1}}}$ and there are $\alpha_k,\beta_k\in \bf_{q^{2^k}}$ such that $\sigma_k^2=\alpha_k\sigma_k+\beta_k$. 
For every $k$, and for every $g\in \sl_{2^n}(q^{2^{k+1}})$ there are are  $g_0,g_1\in M_{2^n}(\bf_{q^{2^{k}}})$ such that $g=g_0+\sigma_k g_1$. We define \[I_{n,k}:M_{2^n}\left(q^{2^{k+1}}\right)\rightarrow M_{2^{n+1}}\left(q^{2^{k}}\right),\quad \text{ as } \quad I_{n,k}(g):=\begin{pmatrix} g_0 & \beta_k g_1 \\ g_1 & g_0 +\alpha_k g_1\end{pmatrix}.\]
It is a straightforward computation to show that for every $g,h\in M_{2^n}(\bf_{q^{2^{k+1}}})$, we have that $I_{n,k}(gh)=I_{n,k}(g)I_{n,k}(h)$. Moreover, if $g\in \sl_{2^n}(q)$, then $\det(I_{n,k}(g))=1$:
 \begin{eqnarray*}\det \begin{pmatrix} g_0 & \beta_k g_1 \\ g_1 & g_0 +\alpha_k g_1\end{pmatrix}&=&\det \begin{pmatrix} g_0 +\sigma_k g_1 & \sigma_k g_0 + (\alpha_k\sigma_k+\beta_k) g_1 \\ g_1 & g_0 +\alpha_k g_1\end{pmatrix} \\
&=&\det \begin{pmatrix} g_0 +\sigma_k g_1 & 0 \\ g_1 & g_0 +(\alpha_k-\sigma_k)g_1\end{pmatrix}
\end{eqnarray*}
by hypothesis $\det(g_0 +\sigma_k g_1)=1$ and $\alpha_k-\sigma_k$ is the Galois-conjugate of $\sigma_k$, hence $\det(g_0 +(\alpha_k-\sigma_k)g_1)=1$. This proves that $I_{n,k}:\sl_{2^{n}}(q^{2^{k+1}})\rightarrow \sl_{2^{n+1}}(q^{2^k})$ is a well-defined group homomorphism. The homomorphism $I_{n,k}$ is also isometric, in fact $v=v_0+\sigma_k v_1\in \ker(g)$ if and only if $(v_0,v_1)\in\ker(I_{n,k}(g))$. Composing the maps 
$$I_{m+n-1,0} \circ \cdots \circ I_{n+1,m-2} \circ I_{n,m-1} \colon \sl_{2^{n}}(q^{2^m}) \to \sl_{2^{n+m}}(q)$$
we obtain an inclusion $\sl_{2^{n}}(q^{2^m}) \hookrightarrow A_0(q)$ that extends the standard inclusion $\sl_{2^n}(q)\hookrightarrow  A_0(q)$. This finishes the proof.
\end{proof}

\begin{rem}
Using the idea in the proof of Proposition \ref{prop:algclos}, one can even show that the inclusion $\sl_{2^n}(q)\hookrightarrow A_0(q)$ extends to $\sl_{2^{n}}(q^{2^{\infty}}) := \cup_n \sl_{2^{n}}(q^{2^m})$ -- a special linear group over an infinite field. The characters of the group $\sl_{2^n}(q^{2^\infty})$ were completely classified by Kirillov \cite{Kirillov1965} for $n>1$ and by the work of Peterson and the second author \cite{Peterson2013} for $n=1$. Any non-trivial irreducible character of $\sl_{2^n}(q^\infty)$ is induced by its center, that is, $\chi(g)=\chi(h)$ for every non-central $g,h\in \sl_{2^n}(q^\infty)$. This can be used in a straightforward way to show that the group $A(q)$ has no non-trivial continuous character. Note that ${\rm Alt}(2^n) \subset \sl_{2^n}(q)$ and that also the work of Dudko and Medynets \cite{Dudko2013} can be invoked to study characters on $A(q)$.

Now, Theorem $2.22$ of \cite{Ando2012} states that any amenable Polish group which admits a complete bi-invariant metric is unitary representable if and only if its characters separate the points. Whence, our observation from above implies readily that $A(q)$ cannot be embedded into any unitary group of a Hilbert space as a topological group. Note that only recently, non-amenable polish groups with a complete bi-invariant metric were found, which are unitarily representable but fail to have sufficiently many characters, see \cite{AMTT}.
\end{rem}

In order to show that $A(q)$ does not have any unitary representation on a Hilbert space, we will show that every continuous positive definite function is trivial. For this, we need the following lemma. 

\begin{lem}
Let $q$ be a prime power and let $\psi \colon \sl_n(q) \to \mathbb C$ be a positive definite function. If there exists a non-central element $g \in \sl_n(q)$ such that $|1 - \psi(x^{-1}gx)|< \varepsilon $ for some $\varepsilon\in(0,1)$ and for all $x \in \sl_n(q)$, then
$$|1 - \psi(h)|<  9(2\varepsilon + 16/q)^{1/2},\quad \forall h \in \sl_n(q).$$
\end{lem}
\begin{proof}
 We set
$$\chi(h) := \frac1{|\sl_n(q)|} \sum_{x \in \sl_n(q)} \psi(x^{-1}hx)$$ 
and note that $\chi \colon \sl_n(q) \to \mathbb C$ is a character.  Hence, we can write
$$\chi(h) = \lambda + \sum_{\pi} \lambda_{\pi} \chi_{\pi}(h),\quad h\in \sl_n(q), $$
where $\pi$ runs through the non-trivial irreducible representations $\pi$ of $\sl_n(q)$, $\chi_{\pi}$ denotes the normalized character of $\pi$, and $\lambda + \sum_{\pi} \lambda_{\pi} = 1$, $\lambda \geq 0$ and $\lambda_{\pi}\geq 0$ for all $\pi$. 
By a result of Gluck \cite[Theorem \textbf{3.4} and Theorem \textbf{5.3}]{Gluck1995}, for every non-central element $h \in \sl_n(q)$ and every non-trivial, normalized, and irreducible character $\chi_{\pi}$ we have $|\chi_{\pi}(h)| < 8/q$. By our assumption $|1 - \chi(g)|< \varepsilon$
and thus
$$\lambda = \chi(g) - \sum_{\pi} \lambda_{\pi} \chi_{\pi}(g) > 1- \varepsilon - 8/q.$$
We conclude that \[|1 - \chi(h)| \leq |\chi(h)-\lambda|+(1-\lambda)\leq 2(1-\lambda)<2\eps+16/q, \quad \forall h \in \sl_n(q).\]
For fixed $h$, the preceding inequality, Markov's inequality and the fact that $\varphi$ is bounded in absolute value by $1$ imply that
$$\left|\left\lbrace x \in \sl_n(q)\: \middle |\: |1 - \psi(x^{-1}hx)| \geq 3(2\varepsilon + 16/q) \right\rbrace \right| \leq \frac{|\sl_n(q)|}{3}$$
and hence that at least $2/3$ of all elements in the conjugacy class of $h$ satisfy
$$|1-\psi(k)| < 3(2\varepsilon + 16/q).$$
Since this holds for all conjugacy classes, we can set
$$A := \left\{k \in \sl_n(q) \mid |1 - \psi(k)| < 3(2\varepsilon + 16/q) \right\}$$
and conclude that $|A| \geq 2/3 \cdot |\sl_n(q)|.$ Therefore for any $h \in \sl_n(q)$, the set $A\cap hA^{-1}\not = \varnothing$ and thus there exist $k_1,k_2 \in A$ such that $h = k_1k_2$.
So the inequality \eqref{eq:ine} yields
$$|\psi(k_1)-\psi(h)|^2\leq 2 (1-Re(\psi(k_2))) \leq 6 (2\varepsilon + 16/q)$$
and hence
$$|1 - \psi(h)| \leq  |\psi(1)-\psi(k_1)| + |\psi(k_1)-\psi(h)| \leq 3(2\varepsilon + 16/q) + (6(2\varepsilon + 16/q))^{1/2} \leq 9(2\varepsilon + 16/q)^{1/2}.$$
This proves the claim.
\end{proof}

Let us now fix a positive definite function $\psi \colon A(q) \to \mathbb C$. Let $\varepsilon>0$ and choose $\delta>0$ such that $\dist(1_{A(q)},x)< \delta$ implies $|1 - \psi(x)|< \varepsilon$. For $n$ large enough, the group $\sl_{2^n}(q)$ will contain a non-central element $g$ with $\dist(1_{A(q)},g)< \delta$. By conjugation invariance of the metric, we conclude that $|1-\psi(xgx^{-1})|< \varepsilon$ for all $x \in A(q)$. Moreover, $\sl_{2^n}(q) \subset \sl_{2^n}(q^m) \subset A(q)$ as explained in Proposition \ref{prop:algclos}. So applying the previous lemma to $\sl_{2^n}(q^m)$, we get that the restriction of $\psi$ to $\sl_{2^n}(q)$ satisfies
$$|1-\psi(h)|< 9(2\varepsilon + 16/q^{m})^{1/2}, \quad \forall h \in \sl_{2^n}(q).$$
Since this holds for any $m \in \mathbb N$, we conclude that $|1 - \psi(h)| \leq 9(2\varepsilon)^{1/2}$ for all $h \in \sl_{2^n}(q)$ and $n$ large enough. Since $\varepsilon>0$ was arbitrary, we must have that the restriction of $\psi$ to any $\sl_{2^n}(q)$ is trivial, and hence by density of $A_0(q)$ in $A(q)$ we can conclude the proof.

\begin{rem}
For the proof of our main theorem, it is not necessary to apply Gluck's concrete estimates. Indeed, it follows already from Kirillov's work on characters of $\sl_n(k)$ for infinite fields $k$ \cite{Kirillov1965} that character values of non-central elements in $\sl_n(q)$ have to be uniformly small as $q$ tends to infinity. This is enough to conclude the proof. One can see the existence of small uniform bounds either by going through the techniques in Kirillov's proof (see also the proof of Theorem 2.4 in \cite{Peterson2013}) or by applying Kirillov's character rigidity to a suitable ultra-product of finite groups $\sl_n(q)$ (for $n$ fixed and $q$ variable) which can be identified with $\sl_n(k)$, where $k$ is the associated ultraproduct of finite fields, which is itself an infinite field.
\end{rem}

\section{A second route to extreme amenability}

\subsection{L\'evy groups}

As outlined in the beginning, every amenable and exotic group is extremely amenable. However, there is a different route to extreme amenability using the phenomenon of measure-concentration. In this section, we want to show that $A(q)$ is a L\'evy group and hence extremely amenable. See \cite{Pestov2006} for more background on this topic. A similar approach was recently carried out in \cite{Dowerk2015a} to give a direct proof that the unitary group of the hyperfinite II$_1$-factor is extremely amenable. As a by-product, we can give explicit bounds on the concentration function that are useful to give quantitative bounds in various non-commutative Ramsey theoretic applications.

\begin{dfn}
  A \textit{metric measure space} $(X,\dist,\mu)$ consists in a set $X$ equipped with a distance $\dist$ and measure $\mu$ which is Borel for the topology induced by the metric, see \cite{Pestov2006}. In the following we will always assume that $\mu$ is a probability measure. Given a subset $A\subset X$ we will denote by $N_r(A)$ the $r$-\textit{neighbourhood} of $A$, i.e., $N_r(A) := \{x \in X \mid \exists y \in A, \dist(x,y)<r\}$. The \textit{concentration function} of $(X,\dist,\mu)$ is defined as
  \begin{align}
    \label{eq:dfnalfa}
    \alpha_{(X,\dist,\mu)}(r):=\sup\left\lbrace 1-\mu(N_r(A))\ \mid \ A\subset X,\ \mu(A)\geq \frac{1}{2}\right\rbrace,\ r>0.
  \end{align}
\end{dfn}

\begin{dfn}
  A sequence of metric measure spaces $(X_n,\dist_n,\mu_n)$ with diameter constant equal to $1$,  is a \textit{L\'evy family} if for every $r\geq 0$,
  \[ \alpha_{(X_n,\dist_n,\mu_n)}(r)\rightarrow 0.\]
A Polish group $G$ (with compatible metrid $\dist$) is called a \textit{L\'evy group} if there exists a sequence $(G_n)_n$ of compact subgroups of $G$ equipped with their normalized Haar measure $\mu_n$, such that $(G_n,\dist|_{G_n},\mu_n)$ is a L\'evy family. If this is the case, we say that the measure concentrates along the sequence of subgroups.
\end{dfn}

The relationship with extreme amenability is given by the following theorem which can be found in \cite[Theorem \textbf{4.1.3}]{Pestov2006}.

\begin{thm}
  Every L\'evy group is extremely amenable.
\end{thm}

Our second route to extreme amenability is therefore to show that $(\sl_{2^n}(q))_n$ is a Le\'vy family with respect to the normalized counting measure and the normalized rank-metric.

\begin{thm}\label{thm:extram}
  The normalized counting measure on the groups $\sl_{2^n}(q)$ concentrates with respect to the rank-metric. In particular, $A(q)$ is a L\'evy group and hence extremely amenable. 
\end{thm}

The proof of the theorem is a straightforward application of the following theorem, whose proof can be found in \cite[Theorem \textbf{4.5.3}]{Pestov2006} or \cite[Theorem \textbf{4.4}]{Ledoux2001}

\begin{thm}
  Let $G$ be a compact metric group, metrized by a bi-invariant metric $\dist$, and let \[\{\id\}=H_0<H_1<H_2<\dots<H_n=G\] be a chain of subgroups. Denote by $a_i$ the diameter of the homogenous space $H_i/H_{i-1}$, $i=1,\dots,n$, with regard to the factor metric. Then the concentration function of the metric-measure space $(G,\dist,\mu)$, where $\mu$ is the normalized Haar measure, satisfies \[\alpha_G(r)\leq 2 exp\left(-\frac{r^2}{16\sum_{i=1}^n a_i^2}\right).\]
\end{thm}

\begin{proof}[Proof of Theorem \ref{thm:extram}]
  Let us fix a base $\{e_1,\ldots,e_n\}$ of $\bf_q^n$. Using the previous theorem, it is enough to show that the diameter of $\sl_n(q)/\sl_{n-1}(q)$ is smaller than $\frac{2}{n}$, where $\sl_{n-1}(q)<\sl_n(q)$ is the subgroup fixing $e_n$. For this, it is enough to show that for every $g\in \sl_n(q)$, there exists $h\in \sl_{n-1}(q)\subset \sl_n(q)$, such that $\rank(hg-\id)\leq 2$. Let $V$ be a 2-dimensional vector space containing $e_n$ and $ge_n$ and let $V'$ be a complement. There exists $h' \in SL(V)$ with $h'(ge_n) = e_n$. Hence, $(h' \oplus 1_{V'})ge_n = e_n$ and thus $(h' \oplus 1_{V'})g \in \sl_{n-1}(q)$. Moreover, 
$$\dist(g,(h' \oplus 1_{V'})g) = \frac 1 n \rank(g-(h' \oplus 1_{V'})g)=\frac 1 n \rank((h' - 1_V) \oplus 0_{V'}) \leq \frac{2}{n}.$$ This proves the claim.
\end{proof}

\subsection{Ramsey theory}
The crucial concept behind the strategy in the previous proof is the notion of \emph{length} of a metric measure space, which is bounded by the value $(\sum_i a_i^2)^{1/2}$ as above and essentially the infimum over these quantities. The formal definition is quite technical and we invite the interested reader to check Definition 4.3.16 in \cite{Pestov2006}. The following lemma is immediate from our computation.

\begin{lem} The length of the metric measure space $(\sl_n(q),d,\mu_n)$ is at most $2n^{-1/2}$. \end{lem}

Any estimate on the length of a metric measure space can be used directly to get explicit bounds on the concentration phenomena, such as in the following standard lemma, whose proof is inspired by the proof of Theorem 4.3.18 in \cite{Pestov2006}.

\begin{lem}\label{lem:ineq}
Let $(G,d,\mu)$ be a metric measure space of length $L$ and let $\eps$ be a positive real. Then for any measurable subset $A\subset G$ satisfying $\mu(A) > 2 e^{-\varepsilon^2/L^2}$ we have
$$\mu(N_{4\varepsilon}(A)) \geq 1-2 e^{-\varepsilon^2/L^2}.$$
\end{lem}
\begin{proof}
Consider the function $d_A:G\rightarrow \br$ defined by $d_A(x):=d(A,x) = \inf\{ d(a,x) \mid a \in A \}$ and set $\lambda := \int_\mu d_A$. Since $d_A$ is $1$-Lipschitz, Lemma 4.3.17 in \cite{Pestov2006} implies that 
$$\mu\left(\{x \in G \mid |d_A(x) - \lambda| \geq 2\eps \}\right) \leq 2 e^{-\eps^2/L^2}.$$

If $\lambda \geq 2\eps$, then for every $x\in A$ we have that $|d_A(x)-\lambda|=\lambda\geq 2\eps$ and the above inequality would give us a contradiction. Therefore we must have that $\lambda\leq 2\eps$ and 
\begin{align*}
  1- \mu(N_{4\varepsilon}(A)) = &\mu\left(\{x \in G \mid d_A(x) \geq 4\eps \}\right)\\
  \leq &\mu\left(\{x \in G \mid |d_A(x) - \lambda| \geq 2\eps \}\right) \leq 2 e^{-\varepsilon^2/L^2}.\qedhere
\end{align*}
\end{proof}

Let us finish this section by applying the previously obtained bounds to deduce an explicit metric Ramsey theoretic result for the finite groups $\sl_n(q)$.

As usual, we say that a covering $\mathcal U$ of a metric space $X$ has \textit{Lebesgue number} $\varepsilon>0$, if for every point $x\in X$ there exists an element of the cover $U\in\mathcal U$ such that the $\varepsilon$-neighborhood of $x$ is contained in $U$.  A covering that admits a positive Lebesgue number $\eps$ will be called a uniform covering or $\eps$-covering. Uniform coverings have been studied in the context of amenability, extreme amenability, and Ramsey theory in \cite{Schneider2015, Schneider2016}. Our main result in this section is a {\it quantitative} form of the metric Ramsey property that is satisfied by the finite groups $\sl_n(q)$ -- resembling the fact the $A(q)$ is extremely amenable.

\begin{thm} \label{thm:rams}
Let $\varepsilon>0$, $q$ be a prime power, and $k,m \in \mathbb N$. If we set $$N:= 64\eps^{-2}\max\{\log(2k),\log(2m)\}$$ then for any $n > N$ and any $\varepsilon$-cover $\mathcal U$ of $\sl_n(q)$ of cardinality at most $m$ the following holds: for every subset $F \subset \sl_n(q)$ of cardinality at most $k$, there exists $g \in \sl_n(q)$ and $U\in \mathcal U$ such that $gF\subset U$. 

\end{thm}

In order to illustrate the result, consider the case $k=3,m=2$. We may think of a covering by two sets as a coloring of $\sl_n(q)$ with two colors, where some group elements get both colors. This covering is uniform if for every element $g\in\sl_n(q)$ the $\varepsilon$-neighborhood of $g$ can be colored by one of the two colors. Now, if $n > 128 \cdot \varepsilon^{-2}$, any subset of three elements $\{a,b,c\} \subset \sl_n(q)$ has a translate $\{ga,gb,gc\}$ colored in the same color. This is in contrast to $(\mathbb R/\mathbb Z,+)$ equipped with its usual metric, since it is easy to see that this group does not admit any uniform covering. 

\begin{proof}[Proof of Theorem \ref{thm:rams}]
Let $\mathcal U$ be the covering and  for every $U \in \mathcal U$ we set $U_0 := \{x \in \sl_n(q) \mid B(x,\varepsilon) \subset U\}$. By assumption, the collection $\mathcal V:=\{U_0 \mid U \in \mathcal U \}$ still forms a covering of $X$. If the cardinality of the covering is at most $m$, there must be one element $V \in \mathcal V$ in the new covering with $\mu_n(V)\geq \frac1m$. Take $U \in \mathcal U$ such that $U_0=V$ and observe that the $\eps$-neighborhood of $V$ is also contained in $U$. If $n> 64\log(2m)\eps^{-2}$ then $1/m > 2e^{-\varepsilon^2 n/64}$ and hence Lemma \ref{lem:ineq} implies that 
\[\mu(U)\geq \mu(N_\eps(V))\geq 1-2e^{-\frac{\eps^2}{16L^2}}=1-2e^{-\frac{\eps^2n}{64}}.\]

For a subset $F\subset \sl_n(q)$ we have
$$\left\{g \in \sl_n(q)\mid gF \subset U \right\} =\bigcap_{h \in F} \left\{g \in \sl_n(q)\mid gh \subset U \right\} = \bigcap_{h \in F} Uh^{-1}$$
and thus if $|F|\leq k$, we get that $\mu(\{g \in \sl_n(q) \mid gF \subset U \}) \geq 1 - 2 k \exp(-\varepsilon^2n/64)$.
Hence, some element $g \in \sl_n(q)$ as desired will exist as soon as
$n > 64\eps^{-2}\max\{\log(2k),\log(2m)\}.$
\end{proof}

\section{Topological simplicity of the central quotient}

In this last section, we will determine the center of $A(q)$ and we will prove that $A(q)$ is topologically simple modulo its center. 

As we remarked in the introduction $A(q)$ is isomorphic to the group of invertible elements of $M(\bf_q)$. Therefore for every $\alpha\in \bf_q^{\times}$ the ``diagonal matrix'' whose non-zero entries are equal to $\alpha$ is an element of $A(q)$ and it is in the center. Let us denote by $Z\subset A(q)$ the group of diagonal matrices with constant values, $Z\cong  \bf_q^{\times}$. We claim that $Z$ is actually the center of $A(q)$ and that $A(q)/Z$ is topologically simple. Note that the quotient $A(q)/Z$ can be understood as the completion of the metric inductive limit of the corresponding projective linear groups with respect to the projective rank-metric as studied in \cite{Stolz2014}.

To prove the claim, we will need the following theorem of Liebeck and Shalev which follows from \cite[Lemma \textbf{5.4}]{Liebeck2001}, see also \cite{Stolz2014} for a discussion of the results of Liebeck-Shalev in the context of length functions on quasi-simple groups.

\begin{thm}[Liebeck-Shalev]\label{thm:LS}
  There exists $c\in \mathbb N$ such that for every $n\in\bn$ and for every $g\in \sl_n(q)$ with $\delta:=\min\{\dist(g,z) \mid z \in \bf_q^{\times} \}>0$, we have $C_{\sl_n(q)}(g)^{\lceil c/\delta\rceil}=\sl_n(q)$.
\end{thm}

For an element $g$ of the group $H$, we denote by $C_H(g)$ its conjugacy class. As a corollary of this theorem we obtain the following. 

\begin{prop}\label{prop:center}
  For every $g\in A(q)$ such that $\delta=\min\{\dist(g,z) \mid z \in \bf_q^{\times} \}>0$, the set $C_{A(q)}(g)^{\lceil c/\delta \rceil}\subset A(q)$ is dense.  
\end{prop}

Before proving the propostion, let us remark that the proposition implies our claim. Let $N<A(q)$ be a closed normal subgroup which contains strictly $Z$ and take $x\in N\setminus Z$. By the previous proposition, there exists $m\in\bn$ such that $C_{A(q)}(x)^m$ is dense, which implies that $N$ is dense and therefore $N=A(q)$. 

\begin{proof}[Proof of Proposition \ref{prop:center}]
  For every $\varepsilon< \delta/2$, we consider an element $g' \in A_0(q)$ with $\dist(g,g') \leq \varepsilon$. Let $n_0$ be such that $g' \in \sl_{2^{n_0}}(q)$. By Theorem \ref{thm:LS} for every $n\geq n_0$ we have $C_{\sl_{2^n}(q)}(g')^{m} = \sl_{2^n}(q)$ for $m:=\lceil c/\delta \rceil$, whence \[C_{A(q)}(g')^m \supset A_0(q).\] This means that for any element $k \in A(q)$, there exist $h_1,\dots,h_m \in A(q)$ such that \[\dist(h_1 g' h_1^{-1} \cdots h_m g' h_m^{-1},k)\leq \varepsilon.\] Since the rank-metric is bi-invariant we obtain that \[\dist(h_1 g' h_1^{-1} \cdots h_m g' h_m^{-1},h_1 g h_1^{-1} \cdots h_m g h_m^{-1})\leq m\dist(g,g'),\] which implies that $\dist(C_{A(q)}(g)^m,k) \leq (m+1)\eps$.
\end{proof}

\begin{rem}
 Since the center of $A(q)$ is $Z=\bf_q^{\times}$, we can conclude that $A(q)$ is isomorphic to $A(q')$ if and only if $q=q'$. One could ask also if the quotient group $A(q)/Z$ depends on $q$. The question seems harder, but one can study centralizer of elements in $A(q)/Z$ as in \cite{Thom2014} to deduce that this group depends at least on the characteristic of the field.  
\end{rem}

\section{On $1$-discrete subgroups of $A(q)$}
\label{disc}

In this section we want to study discrete subgroups of $A(q)$. In fact, we will focus on subgroups $\Gamma \subset A(q)$ so that $\dist(1,g) = 1$ for all non-identity elements of $\Gamma$ -- we call such a subgroup $1$-\emph{discrete} in the metric group $(A(q),\dist)$. Since the diameter of $(A(q),\dist)$ is equal to $1$, $1$-discrete subgroups should play a special role in the study of the metric group $(A(q),\dist)$. Let us start with amenable groups. 

\begin{prop}\label{prop:amena}
  Every countable amenable group is isomorphic to a $1$-discrete subgroup of $A(q)$.
\end{prop}
\begin{proof}
  Let $\Gamma$ be a countable amenable group and suppose at first that there is a sequence of F\o lner sets $(F_n)$ such that $F_{n+1}=F_nD_n$ where $D_n\subset \Gamma$ is a finite subset, so that the family $\{F_nc \mid c\in D_n\}$ consists of pairwise disjoint subsets. For every $h\in \Gamma$ and $n\in\bn$, we define $L_n^h:=\{g\in F_n \mid hg\subset F_n\}$ and we observe that for every $\eps>0$ and $h\in\Gamma$, there exists $N\in\bn$ such that for every $n\geq N$, we have $|L_n^h|>(1-\eps)|F_n|$. We define elements $h_n\in M_{|F_n|}(\bf_q)$ as the usual action by permutation on the left on $L_n^h$ and $0$ on the complement. It is now easy to observe that $(h_n)_n$ is a Cauchy sequence for every $h\in \Gamma$ and hence we can define a maximal discrete embedding of $\Gamma$ into the group of invertible elements of the completion of the inductive limit of matrices of size $|F_n|$, which as explained in the introduction, by the result of Halperin \cite{Halperin}, is isomorphic to $A(q)$.  

For a general countable amenable group we will use Ornstein-Weiss theory \cite{OW}. In fact the proposition follows from a similar argument using as F\o lner sets which are quasi-tiling \cite[page 24, Theorem 6]{OW}. Or one can also observe that $A(q)$ contains the inductive limit of symmetric groups which is the full group of the hyperfinite equivalence relation, which by \cite{OW}, contains any amenable group as maximal discrete subgroup. 
\end{proof}

The aim of this section is to show the following theorem.

\begin{thm} \label{subgroup}
If $q$ is odd, $A(q)$ contains a non-abelian free group as a $1$-discrete subgroup.
\end{thm}

Theorem \ref{subgroup} follows form Elek's work \cite{elek2} and \cite{elek3}. In fact, Elek showed that any amenable skew-field embeds as a discrete sub-algebra of $M(\bf_q)$ and by \cite{Goncalves}, there exist amenable skew-fields with free subgroups. However we do not need the full strength of Elek's results and we will construct an explicit embedding of a skew-field containing a free group into $M(\bf_q)$. 

Note that the analogous result is not true in the setting of II$_1$-factors, i.e., the free group on two generators is not a $1$-discrete (which in this case should mean that different group elements are orthogonal with respect to the trace) subgroup of the unitary group of the hyperfinite II$_1$-factor. Moreover, the preceding result seems to yield the first example of a discrete non-amenable subgroup of an amenable Polish group, whose topology is given by a bi-invariant metric. It is an open problem, if the free group on two generators (or in fact any discrete non-amenable group) can be a discrete subgroup of the unitary group of the hyperfinite II$_1$-factor or the topological full group of the hyperfinite equivalence relation.

The proof of Theorem \ref{subgroup} is based on the following lemma, which is inspired by Elek's canonical rank function, see \cite{elek}.

\begin{lem}\label{lem:inv}
  Let $\Gamma$ be a finitely generated amenable group and suppose that there is a sequence of F\o lner sets $(F_n)$ such that $F_{n+1}=F_nD_n$ where $D_n\subset \Gamma$ is a finite subset. Consider the embedding $\Phi:\Gamma\rightarrow M(\bf_q)$ given in Proposition \ref{prop:amena}. Then for every element $a$ of the group algebra $\bf_q\Gamma$ which is not a zero-divisor, the element $\Phi(a)\in M(\bf_q)$ is invertible.
\end{lem}
\begin{proof}
  Let us fix a non zero-divisor $a\in \bf_q\Gamma$. Let us denote by $S\subset \Gamma$ the support of $a$. For every $n$, as in Proposition \ref{prop:amena},  we define $L_n^a:=\{g\in F_n \mid Sg\subset F_n\}$ and we observe that for every $\eps$, there exists $N\in\bn$ such that for every $n\geq N$, we have $|L_n^a|>(1-\eps)|F_n|$. Observe also that $a$ is the limit of the elements $a_n$ which acts on $L^a_n$ as left translation. We claim that the vectors $\{a_n g\}_{g\in L^a_n}$ are linearly independent. In fact if $g_1,\ldots,g_k\in L_n^a$ and $\alpha_1,\ldots,\alpha_k\in\bf_q$ are such that $\sum_i a(\alpha_i g_i)=0$ then $a(\sum_i \alpha_i g_i)=0$ and whence $a$ is a zero-divisor in $\bf_q\Gamma$. So the rank of $a_n$ is at least $1-\eps$ and hence the rank of $a$ is $1$ and therefore it is invertible. 
\end{proof}

\begin{proof}[Proof of Theorem \ref{subgroup}]
  Let $\Gamma$ be an elementary amenable, torsion free group which satisfies the hypothesis of Lemma \ref{lem:inv}. By \cite[Theorem \textbf{2.3}]{linnell}, every non-zero element of the group algebra $\bf_q\Gamma$ is a not a zero-divisor. By Tamari \cite[Example \textbf{8.16}]{tamari}, we know that group rings of amenable groups which do not contain zero-divisors satisfy the Ore condition, for more about the Ore localization see \cite{stenstrom}. Hence the Ore completion of $\bf_q\Gamma$, denoted by $Q(\bf_q\Gamma)$ is a skew-field. Consider the embedding $\Phi:\Gamma\to M(\bf_q)$ defined in Proposition \ref{prop:amena} and extend it to $\bf_q\Gamma$. Observe that by the universality of the Ore completion and by Lemma \ref{lem:inv}, $\Phi$ extends to a map $\Phi:Q(\bf_q\Gamma)\rightarrow M(\bf_q)$. When $\Gamma$ is also nilpotent, for example when $\Gamma$ is the Heisenberg group, $Q(\bf_q\Gamma)$ has free subgroups by Theorem 5 of \cite{Goncalves}. This concludes the proof.  
\end{proof}

It is plausible that for every maximal discrete embedding of a non amenable group into $M(\bf_q)$ the conclusion of Lemma \ref{lem:inv} cannot hold. 

\section*{Acknowledgements }

This research was supported by ERC Starting Grant No.\ 277728 and ERC Consolidator Grant No.\ 681207. Part of this work was written during the trimester on {\it Random Walks and Asymptotic Geometry of Groups} in 2014 at Institute Henri Poincar\'{e} in Paris -- we are grateful to this institution for its hospitality. We thank G\'abor Elek for interesting remarks that led to the results in Section \ref{disc}.  We also thank the unknown referee for remarks that improved the exposition.


\begin{thebibliography}{15}
\providecommand{\natexlab}[1]{#1}
\providecommand{\url}[1]{\texttt{#1}}
\expandafter\ifx\csname urlstyle\endcsname\relax
\providecommand{\doi}[1]{doi: #1}\else
\providecommand{\doi}{doi: \begingroup \urlstyle{rm}\Url}\fi

\bibitem[Ando and Matsuzawa(2012)]{Ando2012}
Hiroshi Ando and Yasumichi Matsuzawa.
\newblock On Polish groups of finite type.
\newblock \emph{Publ. Res. Inst. Math. Sci.}, 48:\penalty0 389--408, 2012.

\bibitem[Ando et~al.(2016)]{AMTT}
Hiroshi Ando, Yasumichi Matsuzawa, Andreas Thom, and Asger T\"ornquist.
\newblock Unitarizability, Maurey--Nikishin factorization, and Polish groups of finite type.
\newblock \emph{arXiv:1605.06909}, submitted for publication.

\bibitem[Banaszczyk(1991)]{Banaszczyk1991}
Wojciech Banaszczyk.
\newblock \emph{Additive subgroups of topological vector spaces}, volume 1466
  of \emph{Lecture Notes in Mathematics}.
\newblock Springer-Verlag, Berlin, 1991.

\bibitem[Bekka et~al.(2008)Bekka, de~la Harpe, and Valette]{Bekka2008}
Bachir Bekka, Pierre de~la Harpe, and Alain Valette.
\newblock \emph{Kazhdan's Property (T)}.
\newblock Cambridge University Press, 2008.

\bibitem[Dowerk and Thom(2015a)]{Dowerk2015a}
Philip Dowerk and Andreas Thom.
\newblock A new proof of extreme amenability of the unitary group of the hyperfinite II$_1$ factor.
\newblock \emph{Bull. Belg. Math. Soc. Simon Stevin}, 22\penalty0 (5):\penalty0 837--841, 2015.

\bibitem[Dowerk and Thom(2015b)]{Dowerk2015b}
Philip Dowerk and Andreas Thom.
\newblock Bounded Normal Generation and Invariant Automatic Continuity.
\newblock arXiv:1506.08549, submitted for publication.

\bibitem[Dudko and Medynets(2013)]{Dudko2013}
Artem Dudko and Konstantin Medynets.
\newblock On characters of inductive limits of symmetric groups.
\newblock \emph{J. Funct. Anal.}, 264\penalty0 (7):\penalty0 1565--1598, 2013.

\bibitem[Ehrlich(1956)]{Ehrlich1956}
Gertrude Ehrlich.
\newblock Characterization of a Continuous Geometry Within the Unit Group.
\newblock \emph{Trans. AMS}, 83\penalty0 (2):\penalty0 397--416, 1956.

\bibitem{elek}
G\'abor Elek.
\newblock The rank of finitely generated modules over group algebras.
\newblock \emph{Proc. Amer. Math. Soc.} 131\penalty0 (11):\penalty0 3477--3485, 2011. 

\bibitem{elek3}
G\'abor Elek.
\newblock Infinite dimensional representations of finite dimensional algebras and amenability.
\newblock \emph{arXiv:1512.03959}.

\bibitem{elek2}
G\'abor Elek. 
\newblock Convergence and limits of linear representations of finite groups. 
\newblock \emph{J. Algebra} 450\penalty0: 588--615, 2016. 

\bibitem[Galindo(2009)]{Galindo2009}
Jorge Galindo.
\newblock On unitary representability of topological groups.
\newblock \emph{Math. Z.}, 263\penalty0 (1):\penalty0 211--220, 2009.

\bibitem[Gao(2005)]{Gao2005}
Su~Gao.
\newblock Unitary group actions and hilbertian polish metric spaces.
\newblock In \emph{Logic and its applications}, volume 380 of \emph{Contemp.
  Math.}, page 53{\textendash}72. Amer. Math. Soc., Providence, RI, 2005.

\bibitem[Giordano and Pestov(2002)]{Giordano2002}
Thierry Giordano and Vladimir Pestov.
\newblock Some extremely amenable groups.
\newblock \emph{C. R. Math. Acad. Sci. Paris}, 334\penalty0 (4):\penalty0
  273--278, 2002.

\bibitem[Giordano and Pestov(2007)]{Giordano2007}
Thierry Giordano and Vladimir Pestov.
\newblock {Some extremely amenable groups related to operator algebras and
  ergodic theory}.
\newblock \emph{{ J. Inst. Math. Jussieu}}, {6}\penalty0 (02):\penalty0
  279--315, 2007.

\bibitem[Gluck(1995)]{Gluck1995}
David Gluck.
\newblock {Sharper character value estimates for groups of Lie type}.
\newblock \emph{{J. Algebra}}, {174}\penalty0:\penalty0
229--266, 1995.

\bibitem{Goncalves}
Jairo Z. Gon\c{c}alves, Arnaldo Mandel, and Mazi Shirvani.
\newblock{Free products of units in algebras. I. Quaternion algebras}. 
\newblock \emph{J. Algebra} 214\penalty0:\penalty0 301--316, 1999.

\bibitem[Gromov(1993)]{Gromov1993}
Misha Gromov.
\newblock Asymptotic invariants of infinite groups.
\newblock In \emph{Geometric group theory, Vol. 2 (Sussex,
  1991)}, volume 182 of \emph{London Math. Soc. Lecture Note Ser.}, page
  1--295. Cambridge Univ. Press, Cambridge, 1993.

\bibitem{Halperin}
Israel Halperin.
\newblock von Neumann's manuscript on inductive limits of regular rings. 
\newblock \emph{Canad. J. Math.}, 20\penalty0 :\penalty0 477--483, 1968.

\bibitem[Herer and Christensen(1975)]{Herer1975}
Wojchiech Herer and Jens Peter~Reus Christensen.
\newblock On the existence of pathological submeasures and the construction of
  exotic topological groups.
\newblock \emph{Math. Ann.}, 213\penalty0 (3):\penalty0 203--210, 1975.

\bibitem[Kirillov(1965)]{Kirillov1965}
Alexander A. Kirillov.
\newblock Positive-definite functions on a group of matrices with elements from
  a discrete field.
\newblock \emph{Dokl. Akad. Nauk SSSR}, 162:\penalty0 503--505,
  1965.

\bibitem[Ledoux(2001)]{Ledoux2001}
Michel Ledoux.
\newblock \emph{The concentration of measure phenomenon}, volume~89 of
  \emph{Mathematical Surveys and Monographs}.
\newblock American Mathematical Society, Providence, RI, 2001.


\bibitem[Liebeck(2001)]{Liebeck2001}
Martin W. Liebeck and Aner Shalev.
\newblock Diameters of finite simple groups: sharp bounds and applications.
\newblock \emph{Ann. Math.}, 154\penalty0 (2):\penalty0 383--406, 2001.

\bibitem{linnell}
Peter A. Linnell, 
\newblock Noncommutative localization in group rings. 
\newblock \emph{Non-commutative localization in algebra and topology}, Cambridge Univ. Press, Cambridge, 40--59, 2006.

\bibitem[Megrelishvili(2000)]{Megrelishvili2000}
Michael~G. Megrelishvili.
\newblock Reflexively but not unitarily representable topological groups.
\newblock In \emph{Topology Proceedings}, volume~25, 615--625
  (2002), 2000.

\bibitem[Megrelishvili(2001)]{Megrelishvili2001}
Michael~G. Megrelishvili.
\newblock Every semitopological semigroup compactification of the group {$H_+[0,1]$} is trivial.
\newblock \emph{Semigroup Forum}, 63\penalty0(3):\penalty0 357--370, 2001.

\bibitem[vonNeumann(1960)]{vonNeumann1960}
 John von Neumann.
 \newblock \emph{Continuous geometry (foreword by Israel Halperin).} Princeton Mathematical Series 25, 1960.

\bibitem{OW}
Donald S. Ornstein and Benjamin Weiss. 
\newblock Entropy and isomorphism theorems for actions of amenable groups.  
\newblock \emph{J. Analyse Math.}, 48:\penalty0 1--141, 1987.

\bibitem[Pestov(2006)]{Pestov2006}
Vladimir Pestov.
\newblock \emph{Dynamics of infinite-dimensional groups}, volume~40 of University Lecture Series.
\newblock American Mathematical Society, Providence, RI, 2006.

\bibitem[Peterson and Thom(2013)]{Peterson2013}
Jesse Peterson and Andreas Thom.
\newblock Character rigidity for special linear groups.
\newblock \emph{J. Reine Angew. Math.} 716 (2016), 207--228.

\bibitem[Popa and Takesaki(1993)]{Popa2014}
Sorin Popa and Masamichi Takesaki.
\newblock The topological structure of the unitary and automorphism groups of a factor.
\newblock \emph{Comm. Math. Phys.}, 155\penalty0(1):\penalty0 93--101, 1993.


\bibitem[Schneider and Thom(2015b)]{Schneider2015}
Friedrich Martin Schneider and Andreas Thom.
\newblock Topological matchings and amenability.
\newblock \emph{arXiv:1502.02293}, to appear in Fundamenta Mathematicae.

\bibitem[Schneider and Thom(2016a)]{Schneider2016}
Friedrich Martin Schneider and Andreas Thom.
\newblock On F\o lner sets in topological groups.
\newblock \emph{arXiv:1608.08185}, submitted for publication.

\bibitem{stenstrom}
Bo Stenstr\"{o}m.
\newblock \emph{Rings of quotients}. 
\newblock Die Grundlehren der Mathematischen Wissenschaften,
Band 217, An introduction to methods of ring theory, Springer-Verlag, New York, 1975.

\bibitem[Stolz and Thom(2014)]{Stolz2014}
Abel Stolz and Andreas Thom.
\newblock On the lattice of normal subgroups in ultraproducts of compact simple groups.
\newblock \emph{Proc. London Math. Soc.}, 108:\penalty0 73--102, 2014.

\bibitem{tamari}
Dimitri Tamari.
\newblock A refined classification of semi-groups leading to generalized polynomial rings with a generalized degree concept.
\newblock Proceedings of the ICM, Amsterdam 3, 439--440, 1954.

\bibitem[Thom and Wilson(2014)]{Thom2014}
Andreas Thom and John S. Wilson.
\newblock Metric ultraproducts of finite simple groups.
\newblock \emph{C. R. Math. Acad. Sci. Paris}, 352\penalty0(6):\penalty0 463-466, 2014.	

\end{thebibliography}
\end{document}